\documentclass[reqno]{amsart}
\usepackage{amsmath,amsfonts,euscript,amscd,amsthm,amssymb,upref,graphics,color,verbatim}
\usepackage[normalem]{ulem}
\usepackage{enumitem}
\usepackage{graphics}
\usepackage{epsfig}

\theoremstyle{plain}
\newtheorem{theorem}{Theorem}

\newtheorem{lemma}[theorem]{Lemma}

\newtheorem{assumption}[subsection]{Assumption}

\theoremstyle{definition}
\newtheorem{definition}[subsection]{Definition}
\newtheorem{example}[subsection]{Example}

\newtheorem{nothing*}[subsection]{}

\theoremstyle{remark}
\newcommand{\cS}{{\ensuremath{\mathcal{S}}}}
\newcommand{\cB}{{\ensuremath{\mathcal{B}}}}

\newcommand{\cA}{{\ensuremath{\mathcal{A}}}}

\newcommand{\cU}{{\ensuremath{\mathcal{U}}}}
\newcommand{\re}{\mathrm{Re}}
\newcommand{\im}{\mathrm{Im}}
\newcommand{\C}{\mathbb{C}}
\newcommand{\R}{\mathbb{R}}
\newcommand{\N}{\mathbb{N}}
\newcommand{\ra}{\rightarrow}

\title{A RECONSTRUCTION THEOREM FOR COMPLEX POLYNOMIALS}

\author{Luka Boc Thaler}

\begin{document}
\thanks{The author was supported by the research program P1-0291 from ARRS, Republic of Slovenia}
\address{Faculty of Education, University of Ljubljana, Kardeljeva ploščad 16, 1000 Ljubljana, Slovenia}
\address{ Institute of Mathematics, Physics, and Mechanics, Jandranska 19, 1000 Ljubljana, Slovenia.} \email{luka.boc@pef.uni-lj.si}

\begin{abstract}
 Recently Takens' Reconstruction Theorem was studied in the complex analytic setting by Forn{\ae}ss and Peters \cite{FP}. They studied the real orbits of complex polynomials, and proved that for non-exceptional polynomials ergodic properties such as measure theoretic entropy are carried over to the real orbits mapping. Here we show that the result from \cite{FP} also holds for exceptional polynomials, unless the Julia set is entirely contained in an invariant vertical line, in which case the entropy is $0$.

In \cite{T2} Takens proved a reconstruction theorem for endomorphisms. In this case the reconstruction map is not necessarily an embedding, but the information of the reconstruction map is sufficient to recover the $2m+1$-st image of the original map. Our main result shows an analogous statement for the iteration of generic complex polynomials and the projection onto the real axis.
\end{abstract}


%%\renewcommand{\baselinestretch}{1.07}
%%%%%%%%  TOPMATTER:   %%%%%%%%%%%%%%%%%%%%%%%%%
%
%
%\subjclass[2010]{30D05, 37F10, 30D20}
%%\date{April 15, 2013}
%\keywords{Entire functions, Fatou components, wandering domains}
%
%\vfuzz=2pt
%%%%%%%%%%%%%%%%%%%%%%%%%%%%%%%%%%%%%%%%%%%%%%%%%%%%%%%%%%%%%%%%%%%%
%%%%%%%%%%%%%%%%%%%%%%%%%%%%%%%%%%%%%%%%%%%%%%%%%%%%%%%%%%%%%%%%%%%%
%%%%%%%%%%%%%%%%%%%%%%%%%%%%%%%%%%%%%%%%%%%%%%%%%%%%%%%%%%%%%%%%%%%%
%
%\vskip 1cm
%
\maketitle
%%\tableofcontents

\section{Introduction}
In science dynamical systems can be represented, using mathematical methods, with equations which describe the evolution of the system over time. Understanding the dynamics from these equations may be difficult since they may rely on many different parameters. In order to simplify the problem we can suppress some of the parameters and obtain the results from this new equations. It is important to ask whether results obtained this way resemble the original dynamics.

Takens studied this kind of questions in \cite{T1} where he has proved the following theorem.

\begin{theorem}\label{thmtak1}
Let $M$ be a compact manifold of dimension $m$. For an open dense set of pairs $(\varphi, y)$, $\varphi: M \rightarrow M$ a $\mathcal{C}^2$ diffeomorphism and $y: M \rightarrow \mathbb R$ a $\mathcal C^2$ function, the map $\Phi_{(\varphi, y)}: M \rightarrow \mathbb R^{2m+1}$, defined by
$$
\Phi_{(\varphi, y)}(x) = (y(x), y(\varphi(x)), \ldots , y(\varphi^{2m}(x)))
$$
is an embedding.
\end{theorem}

Hence all information about the original dynamical system can be retrieved from the suppressed dynamical system. This theorem no longer holds if we do not assume that the map $\varphi$ is injective. In \cite{T2} Takens proved the following theorem.

\begin{theorem}\label{thmtak2}
Let $M$ be a compact manifold of dimension $m$ and take $k>2m$. For an open dense set of pairs $(\varphi, y)$, $\varphi: M \rightarrow M$ a smooth endomorphism and $y: M \rightarrow \mathbb R$ a smooth function,  there exist  a map $\tau:\Phi_{(\varphi, y)}(M)\rightarrow M$  such that
 $$\tau\circ\Phi_{(\varphi, y)}=\varphi^{k-1},$$
 where  $\Phi_{(\varphi, y)}: M \rightarrow \mathbb R^{k}$  is defined as
$$
\Phi_{(\varphi, y)}(x) = (y(x), y(\varphi(x)), \ldots , y(\varphi^{k-1}(x))).
$$
Moreover if $k>2m+1$ then the set $\Phi_{(\varphi, y)}(M)$ completely determines the deterministic structure of the time series produced by the dynamical system.
\end{theorem}

Recently Forn\ae ss and Peters \cite{FP} studied in how far the dynamical behaviour of complex polynomials can be deduced from knowing only their real orbits. They proved that for non-exceptional polynomials, measure theoretical entropy can be recovered from the real part of the orbits. In this paper we extend their result to the set of all polynomials:

\begin{theorem}\label{thm2} Let $P(z)$ be a complex polynomial of degree $d\geq 2$ and $\nu=\Phi_*(\mu)$, where $\mu$ is the equilibrium measure for $P(z)$. Then the probability measure $\nu$ is invariant and ergodic. Moreover $\nu$ is the unique measure of maximal entropy $\log d$, except when the Julia set of $P$ is contained in invariant line, then its entropy equals zero.
\end{theorem}

The main difficulty here is that in the exceptional case the induced map on the real orbits map does not extend continuously to the natural compactification. The background and proof of Theorem \ref{thm2} are given in sections 2 and 3.

\medskip

The following theorem, whose proof will be given in the last section, is the main result. It is in the same spirit as Theorem \ref{thmtak2} but for complex polynomials and the standard projection to $\R$.

\begin{theorem}\label{thm6}
For a generic  polynomial $P$ of degree $d\geq 2$ there exists $M,N\in\N$ and a map $\tau:\Phi(\C)\rightarrow \C$  such that
$$
\tau\circ\Phi=P^M,
$$
where $$\Phi(z) = (\re(z), \re(P(z)), \ldots , \re(P^{N}(z))).$$
\end{theorem}

By {\it generic} we always mean a countable intersection of open dense sets.  One may ask if there exist an upper bound for $M$, $N$ which is independent of the polynomial. If we would consider all polynomials, then the answer would obviously be  negative. On the other hand there might exist an upper bound which works for a generic polynomial. If such a bound exists, it would be very interesting to find the lowest possible upper bound.

\section{Entropy}

In this section we recall the definition and some basic properties of measure theoretical entropy. The reader is referred to \cite{PY}.

Let $(X,\cB,\mu)$ be a probability space and $\cU=\{A_i\}$ a finite partition, this means that $\mu(X\backslash\cup_i A_i)=0$ and $\mu(A_i\cap A_j)=0$ for all $i\neq j$. Let  $F:X\rightarrow X$ be a $\mu$-preserving map  (i.e. $\mu$ is $F$ invariant).
Let us define the $n$-th partition as $$\cU^n:=\bigvee^{n-1}_{k=0}F^{-k}(\cU).$$
By the definition every set $A\in\cU^n$ is of the form
$$A=A_0\cap F^{-1}(A_1)\cap\ldots\cap F^{-n+1}(A_{n-1}),$$
where $A_k\in\cU$.
Measure theoretical entropy $h_{\mu}(F)$ is now defined as
\begin{equation}\label{ent1}h_\mu(F)=\sup_{\cU}\left(\lim_{n\rightarrow\infty}-\frac{1}{n}\sum_{A\in\cU^n}\mu(A)\log\mu(A)\right).\end{equation}

\begin{lemma}\label{lem0}Let $(X_1,\cB_1,\mu_1)$ and $(X_2,\cB_2,\mu_2)$ be probability spaces and let $T_i:X_i\rightarrow X_i$ be  $\mu_i$-preserving maps. Suppose there is a surjective map $\Phi:X_1\rightarrow X_2$ with the properties $\Phi\circ T_1=T_2\circ \Phi$ and $\mu_2=\Phi_*(\mu_1)$, then $h_{\mu_2}(T_2)\leq h_{\mu_1}(T_1)$.
\end{lemma}

In [B1] Bowen proposed the following slightly different definition of measure theoretic entropy.

Let $X$ be a compact topological space, $F:X\rightarrow X$ a continuous map and $\mu$ an $F$-invariant ergodic probability measure.  Let $U$ be a neighborhood of the diagonal in $X\times X$, and let us define
$$B(x,n,U)=\{y\in X\mid  \{(x,y),(F(x),F(y)),\ldots,(F^{n-1}(x),F^{n-1}(y))\}\subset U\}.$$
The entropy function
\begin{equation}\label{ent2}k_{\mu}(x,F)=\sup_U\left(\lim_{n\rightarrow\infty}-\frac{1}{n}\log(\mu(B(x,n,U)))\right).\end{equation}
 is constant $d\mu$-almost everywhere, and the measure theoretic entropy
is defined to be this constant. When $X$ is a metric space, the sets $B(x,n,U)$  can be replaced by $(n,\varepsilon)$-balls
$$B(x,n,\varepsilon)=\{y\mid d(F^k(x),F^k(y))<\varepsilon,k<n \}.$$

When $X$ is compact,  $F:X\rightarrow X$ a continuous map and $\mu$ an $F$-invariant ergodic probability measure we have $k_{\mu}(Q)=h_{\mu}(Q)$ (see \cite{BK}).
\medskip

The next lemmas are standard results in this topic and their slightly modified versions can be found in \cite{FP}.

Let $P$ be a self map on measurable space $X$ and let $Q$ be a self map on measurable space $Y$ such that there exist a surjective map $\Phi:X\rightarrow Y$ with a property $$\Phi \circ P=Q\circ\Phi.$$

\begin{lemma}\label{lemma1}
If $\lambda$ is invariant on $X$, then the push-forward $\nu$ is invariant on $Y$. If $\lambda$ also is ergodic, then the push-forward, $\nu,$ is ergodic as well.
\end{lemma}

\begin{lemma}\label{lemma3} Let $\nu$ be invariant ergodic measure on $Y$, $E\subset Y$ with $\nu(E)=0 $, and  $N\in \mathbb{N}$. Suppose that $\#\{\Phi^{-1}(y)\}\leq N$ for any $y\in Y\backslash E$.  Then $\nu$ is the push forward of an invariant ergodic measure on $X$.
\end{lemma}

\begin{lemma}\label{lemma2} Suppose we have the setting of Lemma \ref{lemma3}. Let $\mu$ be the unique invariant ergodic measure of maximal entropy on $X$. If $h_{\mu}(P)=h_{\nu}(Q)$,  then $\nu=\Phi_*(\mu)$ and it is the unique measure of maximal entropy on $Y$.
\end{lemma}

\section{Exceptional polynomials}

Given a complex polynomial $P$ let us define the set of real orbits as $$\cA:=\left\{\left(\re(z),\re(P(z)),\re(P^2(z)),\ldots\right)| z\in \mathbb{C}\right\},$$ and observe that the map $P$ naturally defines a shift map $$\rho:\cA\rightarrow\cA.$$
The following lemma is the crucial result of \cite{FP}.
\begin{lemma}\label{lemFP}There exists an $N\in\mathbb{N}$ with the property: If $\re(P^k(z))=\re(P^k(w))$ holds for every $0\leq k\leq N$, then it holds for all  $k\geq0$.
\end{lemma}
Hence  $\cA$ can be identified with a two dimensional subset of $\mathbb{R}^{N+1}$. Let us define a map
$\Phi:\mathbb{C}\rightarrow \mathbb{R}^{N+1}$  as $$\Phi(z)=\left(\re(z),\re(P(z)),\ldots,\re(P^N(z))\right),$$
and let us denote its image by $\cS$.  The polynomial $P$ induces a map $Q$ on $\cS$ with the property $$\Phi \circ P=Q\circ \Phi.$$
Obviously the dynamics of $(\cA,\rho)$ is completely determined by the dynamics of $(\cS,Q)$. For a dense and open set of polynomials the map $Q$ is continuous, but there are some cases where this is not true. See \cite{FP} for the following theorem.
\begin{theorem}\label{thm28}If $P$ is exceptional, then  the image $\Phi(\mathbb C)$ is not closed in $\mathbb R^{N+1}$, hence  $Q$ is not globally continuous map.
\end{theorem}
 The following two types of polynomials are called exceptional.
\begin{definition} A complex polynomial $P(z)$ is strongly exceptional if it maps a vertical line to itself.
\end{definition}

\begin{definition} A complex polynomial $P(z)=a_d z^d+\sum_{k\leq d-2}a_kz^k$ is weakly exceptional if  $a_di^{d-1}$ is real but there is least one $0\leq k\leq d-2$ for
which $a_ki^{k-1}$ is not real.
\end{definition}
We say $P$ is non-exceptional if it is not strongly or weakly exceptional. The following theorem is the main result of \cite{FP}.
\begin{theorem}\label{thm1}Let $P$ be a non-exceptional  complex polynomial of degree $d\geq 2$ and $\nu=\Phi_*(\mu)$, where $\mu$ is the equilibrium measure for $P(z)$. Then the probability measure $\nu$ is invariant and ergodic. Moreover it is the unique measure of maximal entropy, $\log d$.
\end{theorem}

The problem for generalizing this theorem lies in Theorem \ref{thm28}. For exceptional polynomials the set $\cS$ is not closed and $Q$ is not continuous. On the other hand $Q$ is still continuous when restricted to the $\Phi(U)$, where $U\subset\mathbb{C}$ is bounded (see the proof of Theorem 2.8 in \cite{FP}). By observing that in this case $k_{\nu}(Q)$ is well defined and that we have an equality between $k_{\nu}(Q)$ and $h_{\nu}(Q)$, one can slightly modify the original proof of Forn{\ae}ss and Peters to obtain a full result.

\begin{lemma}\label{lem2} Let $P(z)$ be a strongly exceptional holomorphic polynomial of degree $d\geq2$ whose Julia set is contained in the invariant vertical line. Then the probability measure $\nu$ is
invariant and ergodic. The measure $\nu$ is a Dirac measure at the origin, hence $h_{\nu}(Q)=0$.
\end{lemma}
\begin{proof} The measure $\nu$ is invariant and ergodic by Lemma \ref{lem0}. It is also clear that $\nu$ is supported only in the point $\{0\}$. We shall compute the measure theoretic entropy using the  definition (\ref{ent1}). Take any open cover $\cU$ of $\cS$ and denote by $M$ the number of sets $\cU$ containing $0$. There will be only $M$ sets in any $\cU^n$ containing $0$, hence  one concludes that $h_{\nu}(Q)=0$.
\end{proof}

\begin{lemma}\label{lem1} Let $P(z)$ be a holomorphic polynomial of degree $d\geq2$ with an invariant vertical real line  $L\supset P(L)$. If  $\mu$ denotes invariant ergodic measure, then  $\mu(L)=0$ or $\mu(L)=1$.
\end{lemma}
\begin{proof} Let $A=\mathbb{C}\backslash L$. Since $\mu$ is ergodic we need to show that $\mu(A\backslash P^{-1}(A))=0$ and $\mu(P^{-1}(A)\backslash A)=0$. Observe that $P^{-1}(A)\subset A$, hence we only need to prove the first equality. Since $\mu$ is also an invariant measure we obtain $\mu(A)=\mu(P^{-1}(A))$, hence the first equality holds and by ergodicity  $\mu(A)$ is either 0 or 1.
\end{proof}

A set $\{(z,w)\in\mathbb{C}^2\mid \re(P^k(z))=\re(P^k(w)) \quad\forall k\geq0\}$ is called the {\it mirrored set}. We say that $w$ {\it mirrors} $z$ iff $(z,w)$ belongs to the mirrored set. A point $z$ is called a {\it mirrored point} iff there exist a point $w\neq z$ that mirrors it.

\begin{proof}[Proof of Theorem \ref{thm2}] Let $J_P$ denote the Julia set of $P(z)$ and let $J_Q=\Phi(J_P)$. We can naturally obtain $J_Q$ from knowing $(\cS,Q)$, by taking the limit set of the set of all preimages of generic point in $\cS$, hence it is compact. To prove this, take a generic point $x\in\cS$ and take any point $z\in \Phi^{-1}(x)$. We know that the limit set of all preimages of $z$ under $P$ equals $J_P$, hence from $\Phi \circ P=Q\circ \Phi$ we can conclude that the limit set of preimages of $x$ under $Q$ equals to $\Phi(J_P)$. This means that we can obtain $J_Q$ directly from the given data, without knowing $J_P$ and $P$.

 The map $Q:J_Q\ra J_Q$ is a continuous map (using subspace topology in $\mathbb{R}^{N+1}$). Since $\nu$ is a push-forward of $\mu$ it is supported on the compact set $J_Q$ and by Lemma \ref{lemma1} it is invariant and ergodic. Under this conditions, see [BK], we get following equality
$$h_{\nu}(Q)=k_{\nu}(Q).$$

Let us assume that $P$ is strongly exceptional with an invariant vertical line $L$, and that $J_P$ is not contained in $L$, hence $\mu(L)=0$. The invariance and ergodicity of $\nu$ is given by the Lemma \ref{lemma1}. If we prove  that $h_{\nu}(Q)=\log d$, then Lemmas \ref{lemma3}, \ref{lemma2} will finish the proof.

Let $X\subset \cS$ denote the finite set of points $\Phi(z)$ where $z$ is either an isolated mirror point, a singular point of the one dimensional mirror set, an isolated cluster point on the diagonal for mirrored points, or a critical point in either the Julia set or a parabolic basin. We include in $X$ the full orbit for any periodic point in $X.$

Let $Y'$ be the union of the vertical lines passing through the points in $X$. By Lemma \ref{lem1} we get $\mu(Y'\cup L)=0$. Let us denote $Y=Y'\cup L$.

Let $k>>1$ be an integer. Since $\mu(Y) = 0$, there exists an $\delta_0>0$ so that $N_{\delta_0}(Y))$, the $\delta_0$-neighborhood of $Y$, satisfies $\mu(N_{\delta_0}(Y)) < \frac{1}{k}$. Let $z$ now be a \emph{generic point} in $J_P$, as defined in \cite{B2}. Then it follows that
$$
\lim_{n \rightarrow \infty} \frac{\#\{ 0 \le j < n \mid P^j(z) \in N_{\delta_0}(Y) \}}{n} < \frac{1}{k}.
$$
As $Y$ has measure $0$, we may also assume that the orbit of $z$ never hits the set $Y$.

Write $\overline{x} = (x_0, \ldots, x_N) = \Phi(z)$, and let us estimate the entropy function for $\Phi_*(\nu)$ at $\overline{x}$.

For $\ell < n$ define balls in $J_P$,
$$
B'(n,\ell,\epsilon):= \bigcap_{r=0}^{n-\ell}\{w\in J_P: \rho(\Phi(P^r(w)),\overline{x}_{\ell+r})<\epsilon\}.
$$
Here we use the metric $\rho(\overline{x},\overline{y})=
\max_{0\leq i\leq N} |(x_i-y_i)|$ on a large ball in $\mathbb R^{N+1}$ containing the image of the filled-in Julia set.

The idea is to show that for $\epsilon>0$ small enough,
\begin{equation}\label{eq:division}
\mu(B'(n,\ell-1,\epsilon))\leq \frac{\mu(B'(n,\ell,\epsilon))}{d},
\end{equation}
for all $\ell$ except at most a fraction of $1/k$, and for those we have
$$
\mu(B'(n,\ell-1,\epsilon))\leq \mu(B'(n,\ell,\epsilon)).
$$
It follows that the $\mu$-measure of $B'(n,0,\epsilon)$ is at most $C(1/d)^{n-\frac{n}{k}}.$ Therefore the metric entropy of $\nu$ is at least $\log d.$
To obtain inequality (\ref{eq:division}) one can simply follow the proof of Lemma 4.10 from \cite{FP}.

When $P$ is weakly exceptional simply take $L=\emptyset$. Lemma \ref{lem2} now completes the proof of our theorem.

\end{proof}

\section{A reconstruction theorem for polynomials}
 
The aim of this section is to prove Theorem \ref{thm6}. In order to do so we will need a better understanding of a mirrored set and what happens with mirrors under the iteration.

\begin{definition} Suppose that $z$ mirrors $w$. We say that the {\it mirror breaks} at time $n$  if 
$P^k(z)\neq P^k(w)$ for $k<n$ and $P^n(z)= P^n(w)$.
\end{definition}

 We first observe that for a polynomial $P$ of degree $d\geq 2$, the map $$z\ra(\re(z), \re(P(z)),\re(P^{2}(z)),\re(P^{3}(z)),\ldots),$$
is not injective. This follows from the following simple observation. Take any point $z$ close to the infinity and observe its preimages $z_k\in P^{-1}(z)$. If $\re(z_k)\neq\re(z_j)$ for all $j\neq k$, then by moving our initial point $z$ along the level curve of the Greens function, we can achieve that two preimages satisfy $\re(z_k)=\re(z_j)$. Hence $z_j$ mirrors $z_k$ which proves that the set  mirrored points is never empty.
 In this example the mirror between $z_k$ and $z_j$ breaks at time 1 since $P(z_j)=P(z_k)$, therefore we may ask if all mirrors break and if they do, is there an upper bound on the time. Observe that in general mirrors do not need to brake. For example observe that for a real polynomial $P$ a point $z$ is mirrored by $\bar{z}$ and generic mirrors never break. Another example would be a polynomial with two fixed points lying one above the other. On the other hand the Theorem \ref{thm6} states that for a generic polynomial $P$  all mirrors break at most in time $M< \infty$. The following lemma will be used later in the proof.

\begin{lemma}\label{lem18} If $P(z)=a_dz^d+\ldots+a_1z+a_0$ is non-exceptional polynomial and $a_d\notin\R$, then in a small neighborhood of infinity all mirrors break at time 1.
\end{lemma}

\begin{proof} Near infinity there exist a holomorphic change of coordinates of the form $\xi:=z+O(1)$, as $|z|\rightarrow \infty$, that conjugates $P(z)$ to the map $\xi\ra{a_d}\xi^d$. Since $P$ is non-exceptional it follows from Lemma 3.4 in \cite{FP} that any point $z$ near infinity can only be mirrored by a point that lies on the same level curve $\{|\xi|=R\}$ as $z$. Observe that such level curves for $P$ are of form $\{z=R e^{i\varphi}+O(1)\},$
hence near infinity every a mirrored pair $(z,w)$ must satisfy $$w=\bar{z}+O(1).$$

It follows that $$P(w)=a_dw^d+O(w^{d-1})=a_d\bar{z}^d+O(\bar{z}^{d-1}).$$
We also have $$P(z)=a_dz^d+O(z^{d-1}),$$
hence $$P(w)-\overline{P(z)}=(a_d-\bar{a}_d)\bar{z}^d+O(\bar{z}^{d-1}),$$
which converges to infinity as $z\ra\infty$. Hence if $z$ mirrors $w$ then $P(z)=P(w)$.
\end{proof}

In what follows, we will be using only elementary facts about real algebraic and semi-algebraic sets. The reader is referred to the standard text of Bochnak, Coste and Roy \cite{BCR} for
more details.

The set of mirrored pairs of points can be defined in the following way
$$
\tilde{X}=\{(z,w)\mid \re(P^k(z))=\re(P^k(w)) \quad \forall k\geq0\}\subset\mathbb{C}^2.
$$
and as was observed by Forn{\ae}ss and Peters, see \cite[Lemma 2.3]{FP}, there exist an $N$ such that
$$
\tilde{X}=\{(z,w)\mid \re(P^k(z))=\re(P^k(w)) \quad  k\leq N\}.
$$
Note that $\tilde{X}$ is a real algebraic set. The problem with this definition is that the diagonal 
$\Delta:=\{(z,z)\mid z\in\mathbb{C}\}$ is fully contained in $\tilde{X}$ and therefore $\dim \tilde{X}=2$ , see \cite[Lemma 3.10 ]{FP}. As we will see next it is possible to define the mirrored set in  such a way that it contains only finitely many points on the diagonal. Moreover generically these points will be precisely the critical points of the polynomial. 

Let  $P(z)=a_dz^d+\ldots+a_1z+a_0$  and observe that
$$
P(z)-P(w)=(z-w)\left(a_d\sum_{k=0}^{d-1}{z^kw^{d-1-k}}+ a_{d-1}\sum_{k=0}^{d-2}{z^kw^{d-2-k}}+\ldots+a_2(z+w)+a_1\right),
$$
and similarly for $n\geq 1$
$$
P^n(z)-P^n(w)=(z-w)R_n(z,w)
$$
where the polynomial $R_n(z,w)$ satisfies $R_n(z,z)=\frac{d}{dz}P^n(z)$. Next observe that the condition $\re(P^n(z)-P^n(w))=0$ implies
$$
\re(R_n(z,w))\re(z-w)+\im(R_n(z,w))\im(z-w)=0.
$$

Let us define real polynomials
 $$
\begin{aligned}
Q_0(z,w)&=\re(z-w),\\
Q_n(z,w)&=\im\left(\frac{P^n(z)-P^n(w)}{z-w}\right)=\im(R_n(z,w)),
\end{aligned}
$$
for $n\geq 1$ and the set
$$
X=\bigcap_{n\geq0}\{Q_n=0\}.
$$
Observe that $X$ is a real algebraic set contained in $\tilde{X}$ and  that $\tilde{X}\backslash X\subset \Delta$. Moreover for non-exceptional, non-real polynomial $P$ we have $\dim X= 1$, see  \cite[Lemma 3.11.]{FP}.

Next we define for $n\geq 1$  real algebraic sets
 $$Y_n=\{(z,w)\mid Q_k(z,w)=0 \quad \forall 0\leq k<n,\quad R_n(z,w)=0\}.$$
 Observe that these sets satisfy $Y_n\subset Y_{n+1}\subset X$ for all $n\geq 1$, and that Theorem \ref{thm6} is equivalent to $Y_M=X$ for some $M\geq 1$. Therefore our goal is to prove that for a generic polynomial $P$ we can find $M>0$ so that $X\backslash Y_M=\emptyset$. 
 
\begin{proof}[Proof of Theorem \ref{thm6}] Let us start with our fist assumption. 
 \begin{assumption}\label{condition1} Our polynomial $P$ is non-exceptional and its leading coefficient is non-real. Note that this condition is satisfied by an open and dense set of polynomials.
 \end{assumption}
 
First note that $X\backslash Y_n$ are semi-algebraic sets and they satisfy  $X\backslash Y_n\supset  X\backslash Y_{n+1}$ for all $n\geq 1$. Moreover by Lemma \ref{lem18} we know that these sets are bounded. Let us define 
$$S:=\bigcap_{n\geq1} X\backslash Y_n$$
and we have to study two cases.

{\bf Case 1.} If $S=\emptyset$ we claim that $X\backslash Y_M=\emptyset$ for some large $M$. First note  that under this assumption we have the following chain real algebraic sets
$$Y_1\subset Y_2\subset\ldots\subset \bigcup_{n\geq 1}Y_n=X,$$

Recall that every real algebraic set $A$ it is a union of  finitely many irreducible real algebraic sets $A=V_1\cup\ldots\cup V_k$ where $V_j\not\subset V_k$ for $j\neq k$. Note that if $V_1$ and $V_2$ are irreducible real algebraic sets and if $V_1\subset V_2$, then $\dim V_1=\dim V_2$ implies $V_1=V_2$. Since $\dim X$ coincides with the maximal length of the chains of distinct non-empty irreducible real algebraic subsets of $X$  we can see that there must exist $M<\infty$ such that $Y_M=X$ which completes the proof of Theorem \ref{thm6} for this case.   

\medskip 

{\bf Case 2.} Now we assume that  $S\neq\emptyset$. Note that by the definition of $S$ the point $(z,w)\in S$ if and only if  $(z,w)\in X$ and $P^k(z)\neq P^k(w)$ for all $k\geq 0$. It follows that if 
$(z,w)\in S$ then also  $(P^k(z),P^k(w))\in S$ for all $k\geq 1$. With a slight abuse of notation we define a map $P:\mathbb{C}^2\rightarrow \mathbb{C}^2$ by $P(z,w):=(P(z),P(w))$ and note that since $P$ is continuous we have $P(\overline{S})\subset \overline{S}$, where $\overline{S}$ denotes the standard Euclidean closure.  The set $\overline{S}$ is a compact real semi-algebraic set. Our goal is to prove that for a generic polynomial no periodic point can be mirrored, and that for this selection of polynomials the case were $\overline{S} \neq\emptyset$ will imply the existence of a periodic mirrored point which will bring us to the contradiction, hence we must be in Case 1.

We have already made one Assumption \ref{condition1} on our polynomial $P$ and we continue to add several more assumptions:

\begin{assumption}\label{rem1} Our polynomial $P$ satisfies 
$X\cap\Delta=\{(z,z)\mid P'(z)=0\}$ and note that this condition is satisfied by an open and dense set of polynomials.

\emph{ Indeed, this follows from the fact that $X\cap\Delta=\{(z,z)\in\C^2\mid \im(\frac{d}{dz} P^n(z))=0 \text{ for all } n\geq 1\}$.}
\end{assumption}

\begin{assumption}\label{rem3} The critical set of our polynomial $P$ does not contain any periodic cycle.  Note that this condition is satisfied by an open and dense set of polynomials. \end{assumption}

\begin{assumption}\label{rem2} Our polynomial $P$  has no Siegel disks and note that this condition is satisfied by an open and dense set of polynomials.

 \emph{Indeed, suppose that $P$ has a Siegel disk. Then we can take an arbitrarily small perturbation of the polynomial for which the neutral periodic point at the center of the Siegel disk becomes attracting. As attracting periodic points are stable under sufficiently small perturbations and since the number of periodic attracting cycles is bounded by the degree of the polynomial minus $1$, the set of polynomials for which the number of distinct attracting cycles is locally maximal is open and dense, and these polynomials do not have Siegel disks.}
\end{assumption}

\begin{assumption}\label{rem4} Our polynomial $P$ has no mirrored periodic points.  Note that this condition is satisfied by a generic polynomial.

\emph{Indeed, first observe that for any non-exceptional polynomial, a periodic point can only be mirrored by another pre-periodic point which follows from the simple fact that a point can have at most $(\deg P)^2$ mirrors. Next observe that if we conjugate a polynomial $P$ by a rotation, i.e. $P_\varphi(z)=e^{i\varphi}P(ze^{-i\varphi})$, there are only countably many angles $\varphi\in[0,2\pi]$ for which there are two periodic points lying one above the other so for a generic polynomial no two periodic point can mirror each other. Using the same argument can also make sure that a periodic point $z$ does not mirror any of it's pre-images $\{w\mid P(w)=z\}$. This shows that for a generic polynomial no periodic point can be mirrored by a pre-periodic point.}
\end{assumption}

  Let $P$ be a generic polynomial that satisfies all the conditions from our Assumptions \ref{condition1}, \ref{rem1}, \ref{rem3}, \ref{rem2} and \ref{rem4}. Observe that the Assumptions  \ref{rem1} and  \ref{rem3} in combination with the property $P(\overline{S})\subset \overline{S}$ imply that $\overline{S}\cap \Delta=\emptyset$. Next observe that Assumptions  \ref{rem2} and  \ref{rem4} imply that $\overline{S}\subset J\times J$, where $J$ is the Julia set of $P$. Indeed, if $z$ belongs to the Fatou component, i.e. connected component of $\mathbb{C}\backslash J$, and if 
 $(z,w)\in  S$ then there is a subsequence of iterates $(P^{n_k}(z),P^{n_k}(w))\rightarrow (\zeta,\zeta)$ as $k\rightarrow \infty$, where $P(\zeta)=\zeta$. But this implies that $(\zeta,\zeta)\in  \overline{S}$ which can not happen since $\overline{S}\cap \Delta=\emptyset$.

The semi-algebraic set $\overline{S}$  has finitely many connected components and since $P(\overline{S})\subset \overline{S}$ there is a component $V$ of $\overline{S}$ and  $k\geq 1$ such that $P^k(V)\subset V$. The component $V$ is a compact semi-algebraic set which is a triangulable space and recall that Lefschetz fixed point Theorem states that every continuous self-map $f $ of compact triangulable space $X$ with non-zero Lefschetz number
$$\Lambda_f:=\sum_{k\geq0}(-1)^k\mathrm{Tr}(f_*\mid H_k(X,\mathbb{Q}))$$
has a fixed point.

Let $proj_1:\C^2\ra\C$ be a projection to the first coordinate and note that $W:=proj_1(V)$ is a compact semi-algebraic set. If it is contractible, then  $H_k(W,\mathbb{Q})=0$ for $k\geq1$, hence 
$$\Lambda_{P^k}=\mathrm{Tr}(P^k_*\mid H_0(W,\mathbb{Q}))=1.$$ 
It follows that $P^k$ has a fixed point which is a contradiction with our earlier observation that no periodic points are mirrored.

\medskip
Finally suppose that $W$ is not contractible. Then there is a $P^k$-invariant bounded Fatou component $\Omega$ whose boundary  is a subset of $W$. By replacing $W$ with $P^n(W)$ for some $n<k$ if necessary we may assume that $\Omega$ contains a critical point. A map $P^k:\partial \Omega \ra \partial \Omega $ is a holomorphic self-map of a Jordan curve, hence
$$
\begin{aligned}
\Lambda_{P^k}&=\mathrm{Tr}(P^k_*\mid H_0(\partial \Omega,\mathbb{Q}))-\mathrm{Tr}(P^k_*\mid H_1(\partial \Omega,\mathbb{Q}))\\
\Lambda_{P^k}&=1-\mathrm{degree(P^k,\partial \Omega)}\neq0.
\end{aligned}
$$
It follows that $P^k$ has a fixed point in $\partial \Omega\subset W$ which is again a contradiction with our assumption that no periodic points are mirrored. This completes the proof of the Theorem \ref{thm6} in case 2.

\end{proof}
\medskip

The following two examples show that the set of polynomials satisfying Theorem \ref{thm6}  is not open, but it has a non-empty interior.
\medskip

\begin{example} Take  polynomials $P_{\varepsilon,\varphi}(z)=z+e^{i\varphi}(z^2+\varepsilon)$, where $\varepsilon>0$. Observe that for a generic $\varphi\in[0,2\pi]$ the polynomial $P_{0,\varphi}(z)=z+e^{i\varphi}z^2$ satisfies all the conditions in the proof of Theorem \ref{thm6}, hence there exist $P_{0,\varphi}$ for which the Theorem  \ref{thm6} holds. Now observe that $P_{\varepsilon,\varphi}$ converges to $P_{0,\varphi}$ and that every $P_{\varepsilon,\varphi}$ has two fixed points $-i\sqrt{\varepsilon}$  and $+i\sqrt{\varepsilon}$ with the same real parts, hence the set of polynomials  for which the Theorem  \ref{thm6} holds is not open. \end{example}
\medskip

\begin{example} Let $d\geq2$ and $P(z)=iz^d-e^{i\psi}c$ where $c:=c(d)$ is some large positive constant. We will prove that $P$ and any small perturbation of $P$ satisfies Theorem \ref{thm6}.

Let us define $R_c=\sqrt[d]{c+c^{\frac{3}{2d}}}$, $r_c=\sqrt[d]{c-c^{\frac{3}{2d}}}$ and for every $k\in\{0,1,\ldots, d-1\}$ we define
$$V_k=\left\{z\in\C \mid r_c<|z|<R_c,\quad \left|d\arg{z}-\psi-2k\pi-\frac{\pi}{2}\right|<\frac{2R_c}{c} \right\}.$$
where $\frac{2R_c}{c}\rightarrow  0$ when $c\rightarrow \infty$. Now define $$V=\bigcup_{k=0}^{d-1}V_k.$$
 We will see that for sufficiently large $c$ our map $P$ will map a complement of $V$  into the complement of a ball of radius $R_c$ centered at the origin. It is easy to verify that $P$ maps the ball of radius $r_c$ to the complement of the ball of radius $R_c$, and that $P$ maps to the complement of the ball of radius $R_c$ to itself. Observe that
$$|P(re^{i\varphi})|=\sqrt{r^{2d}+c^2+2cr^d\sin{(d\varphi-\psi)}}$$
and that one obtains a minimum when $r^d=-c\sin{(d\varphi-\psi)}$, hence
$$|P(re^{i\varphi})|\geq c|\cos{(d\varphi-\psi)}|.$$
Now observe that if $\left|d\arg{z}-\psi-2k\pi-\frac{\pi}{2}\right|>\frac{2R_c}{c}$, then for sufficiently large $c$ we have $|\cos{(d\varphi-\psi)}|>\frac{R_c}{c}$, and therefore  $|P(re^{i\varphi})|>R_c$. We have proven that  the complement of $V$  is mapped by $P$ to the complement of a ball of radius $R_c$ centered at the origin.

Taking $c$ sufficiently large we can assume that the critical point has an unbounded orbit, hence the Julia set $J$ is totally disconnected and it is contained in $V$. Observe that every $V_k$ contains a fixed point. Let  $\psi=\frac{\pi}{\sqrt{5}}$ or any other irrational angle then by increasing $c$  if necessary we may assume that the projections of $V_k$ to the real line are pairwise disjoint.

Using symbolic dynamics we can uniquely identify every point from the Julia set,  by elements in $\{1, \ldots , d\}^{\mathbb N}$. Suppose $\{a_n\}_{n\geq0}$ is such sequence then there is a unique point $z\in J$ with the property $P^n(z)\in V_{a_n}$. This implies that points in the Julia set can not mirror each other, moreover it follows from Lemma \ref{lem18} that points in Julia set have no mirrors. As a consequence we obtain that the set $\overline{S}$, defined in the proof of the Theorem \ref{thm6}, is an empty set. Hence there exists integer $M>0$ such that every mirror breaks at time $k\leq M$.

By Mane-Sad-Sullivan \cite{MSS} our polynomial $P$ is $J$-stable polynomial when $c$ is large, hence the dynamics of small perturbation of polynomial $P$ resembles the dynamics of initial polynomial $P$, moreover their Julia sets are topologically conjugated. Since above equations are only slightly perturbed when $P$ is perturbed we can conclude that there exists an open set of polynomials satisfying Theorem \ref{thm6}.
\end{example}
\medskip

{\bf Final remarks:} The question remains if Theorem \ref{thm6} also holds for an open and dense set of polynomials. Note that the openness is lost by the Assumption \ref{rem4} which ensures that no periodic point is mirrored. If one can find a better argument which would imply that the set of polynomials satisfying  the Assumption \ref{rem4}  is open and dense then this would give a positive answer to the question above.

It is not to difficult to see that the constant $N$ in the Theorem \ref{thm6} can be chosen to depend only on the degree of a polynomial. This naturally leads to the question if the same holds for the constant $M$. Moreover it would be interesting if we could find their numerical values.

\section*{Acknowledgments}
 The author would like would like to express his gratitude to Han Peters for helpful conversations and comments during the time of research and during the process of writing this article. The author would also like to thank the referee for useful comments that helped improve the clarity and the relevance of this paper. Author was supported by the Slovenian Research Agency (ARRS).

\end{document}